\documentclass[11pt,a4paper]{article}
\usepackage{amsmath}
\usepackage{indentfirst}
\usepackage{geometry}
\usepackage{amsfonts}
\usepackage{amssymb}
\usepackage{amsthm}
\usepackage{color}
\usepackage{mathtools}
\newtheorem{theorem}{Theorem}
\newtheorem{lemma}{Lemma}
\newtheorem{proposition}{Proposition}

\newtheorem{definition}{Definition}

\newtheorem{remark}{Remark}
\newtheorem{assumption}{Assumption}

\newcommand{\be}{\begin{equation}}
\newcommand{\ee}{\end{equation}}
\newcommand{\beq}{\begin{eqnarray}}
\newcommand{\eeq}{\end{eqnarray}}

\newcommand{\ced}{\end{proof}}

\title{Lattice Approximations of Semilinear Stochastic Elliptic Equations with Reflection\footnotemark[1]}
\author{Jun Dai \footnotemark[2] \and Jing Zhang\footnotemark[3]}

\begin{document}
\maketitle
\begin{abstract}
We study lattice approximations of reflected stochastic elliptic equations driven by white noise on a bounded domain in $\mathbb{R}^d,\ d=1,2,3$. The convergence of the scheme is established.
\end{abstract}


\footnotetext[2]{School of Mathematical Sciences, Fudan University, Shanghai 200433, China (\textit{e-mail: 13110180052@fudan.edu.cn}).}
\footnotetext[3]{School of Mathematical Sciences, Fudan University, Shanghai 200433, China (\textit{e-mail: zhang\_jing@fudan.edu.cn}).}
\footnotetext[1]{The work of the second author is supported by National Natural Science Foundation of China (11401108) and Shanghai Science and Technology Commission Grant (14PJ1401500).}

\section{Introduction}
Considering the following reflected stochastic partial differential equation (SPDE for short) of elliptic type with Dirichlet boundary condition:
\begin{equation}\label{eq:Ell1}
 \left\{
  \begin{aligned}
    &-\Delta u(x)=f(x,u(x))+\sigma(x,u(x))\dot{W}(x)+\eta(x),\quad x\in D;\\
    &u(x)=0,\quad x\in\partial D;\quad u(x)\geq 0,\quad x\in D,
  \end{aligned}
 \right.
\end{equation}
where $D:=(0,1)^d$, and $\{\dot{W}(x), x\in \overline{D}\}$ is a white noise on $\overline{D}:=[0,1]^d$. The coefficients $f$ and $\sigma$ are measurable mappings from $[0,1]^d\times\mathbb{R}$ into $\mathbb{R}$, and $\eta(x)$ is a random measure which is a part of the solution $(u,\eta)$ and plays a similar role as local time that prevents the solution $u$ from being negative. In this article, we assume the dimension $d=1,2,3$.

Nualart and Tindel \cite{NUA} studied the elliptic SPDE with an additive noise reflected at zero, when $\sigma(\cdot)=1$ and $D$ is a bounded domain of $\mathbb{R}^d,d=1,2,3$. The stochastic elliptic equation was transformed into a deterministic obstacle problem with irregular boundary, of the type studied by Bensoussan and Lions \cite{BEN} and the references therein. Since the boundary function of the translated problem was not smooth enough, their strong existence and uniqueness result was not covered by the deterministic theory of variational inequalities. Yue and Zhang \cite{WEN} studied the elliptic SPDEs with two reflecting walls driven by multiplicative noise, which extended the results of \cite{NUA}. For the Physical background, the dynamics of the location of the interface near two hard walls is determined by the SPDEs with two reflecting walls, see Funaki and Olla \cite{FUN}.

The discretization scheme for SPDEs of elliptic type was discussed by Gy\"{o}ngy and Martinez \cite{GYO}, when $D=(0,1)^d, d=1,2,3$. For these ranges of dimensions, they introduced a numerical scheme based on the discretization of the Laplacian and gave the rate of convergence in $L^2(D)$-norm and in $L^\infty (D)$-norm. Actually, the lattice scheme in \cite{GYO} is related to truncated Fourier expansions. Later Martinez and Sanz-Sol\`{e} in \cite{MAR} studied a lattice approximation for the elliptic SPDE driven by a coloured noise for $d\geq 4$ extending the results of \cite{GYO}. Actually, for $d\geq 4$, the Green function $G_D(x,\cdot)\in L^\alpha(D),\ \alpha\in [1,d/(d-2))$, uniformly in $x$. In particular, the stochastic integral $\int_D G_D(x,y)dW(y)$ with respect to a white noise cannot be defined as a real-valued $L^2$ random variable. This problem forced the choice of a colored noise instead of a white noise to give a rigorous meaning to \eqref{eq:Ell1} for $d\geq 4$.

For the parabolic SPDEs with reflection, Nualart and Pardoux in \cite{NUA1} studied a nonlinear heat equation on the spatial interval $[0,1]$ with Dirichlet boundary conditions, driven by an additive space-time white noise. Donati-Martin and Pardoux in \cite{DON} generalized the model of \cite{NUA1}. The nonlinearity appears both in the drift and in the diffusion coefficients. They prove the existence of the solution by penalization method but they didn't obtain the uniqueness result. And then in 2009, Xu and Zhang solved the problem of the uniqueness, see \cite{Xu}. Recently, Denis, Matoussi and Zhang \cite{DEN} studied a more generalized model with the term of divergence and the coefficients depending on the gradient of the solution. Their method is based on analytical technics coming from the parabolic potential theory.

The discretization scheme for stochastic heat equations driven by space-time white noise was first introduced by Gy\"{o}ngy in \cite{GYO1} and \cite{GYO2}. Zhang \cite{ZHA} introduced a discretization scheme for reflected stochastic partial differential equations driven by space-time white noise through systems of reflecting stochastic differential equations. He studied the existence and uniqueness of solutions of  Skorohod-type deterministic systems on time-dependent domains to establish the convergence of the scheme. He also established the convergence of an approximation scheme for deterministic parabolic obstacle problems.

The purpose of this paper is to develop a numerical scheme for the reflected SPDEs of elliptic type, extending the results in \cite{ZHA}. As in \cite{ZHA}, part of the difficulties is caused by the discretization of the random measure $\eta$ appeared in the equation \eqref{eq:Ell1}. Besides, part of the difficulties lies in the higher dimension of space. We follow the classical method mentioned by \cite{NUA} and transform the reflected stochastic elliptic equations into a SPDE without reflection and a deterministic elliptic obstacle problem. To prove the convergence of the scheme, we need to establish the existence, uniqueness, continuous dependence with respect to barriers and convergence of a discretization scheme of deterministic elliptic obstacle problems. Finally, with the help of the $L^2$-norm estimates about the Green functions of the original problem and its finite difference approximation (see \cite{GYO}), we get our convergence results.

The organization of the paper is as follows. In Section 2, we lay down the framework. Then we introduce the discretization scheme and the main result in Section 3. In Section 4, we establish the existence, uniqueness, continuous dependence with respect to barriers and convergence of a discretization scheme of deterministic elliptic obstacle problems. The Section 5 is devoted to the proof of the existence, uniqueness and convergence of the discretization scheme for elliptic SPDEs with reflection. We relate the elliptic SPDEs with reflection to a random elliptic obstacle problem and obtain the convergence of the scheme by carefully comparing it with the discretization scheme introduced in Section 3.

\section{Framework}
Let $D:=(0,1)^d$. Let $(\Omega, \mathcal{F}, P)$ be a complete probability space carrying a Brownian sheet $W=\{W(x): x\in \overline{D}\}$, which is a continuous Gaussian random field on $\overline{D}:=[0,1]^d$, satisfying: $EW(x)=0$ and $EW(x)W(y)=x\wedge y=\prod_{j=1}^{d}(x_j\wedge y_j)$ for $x,y\in \overline{D}$.

In \eqref{eq:Ell1}, $\dot{W}(x)$ is the formal derivative of $W$ with respect to the Lebesgue measure and the symbol $\Delta$ denotes the Laplace operator on $L^2(D)$. If $u(x)$ hits 0, additional forces are added in order to prevent $u$ from leaving 0. Such an effect will be expressed by adding extra (unknown) term $\eta$ to \eqref{eq:Ell1}.

Denote by $\langle\cdot,\cdot\rangle$ the dot product of Euclidean space, by $(\cdot,\cdot)$ the scalar product of $L^2(D)$, and by $||\cdot||_\infty$ the supremum norm on $D$. The symbols $\alpha=(\alpha_1,\cdots,\alpha_d)$ and $i=(i_1,\cdots,i_d)$ denote indices from the sets of multi-indices
\begin{equation*}
  I^d:=\{1,2,\cdots\}^d,\quad I_n^d:=\{1,\cdots,n-1\}^d,
\end{equation*}
respectively.

The test function space $C_0^\infty(D)$ is the set of infinitely differentiable functions on $D$ with compact supports.

We assume that $d=1,2$ or $3$. Let $f,\sigma: D\times \mathbb{R}\to\mathbb{R}$ be measurable functions satisfying:
\begin{assumption}\label{ass:a1}
 There exists a Lipschitz constant $L_1$, such that for any $x,y\in D$, $u,v\in\mathbb{R}$,
 \begin{displaymath}
  |f(x,u)-f(y,v)|+|\sigma(x,u)-\sigma(y,v)|\leq L_1[|x-y|+|u-v|].
 \end{displaymath}
\end{assumption}
\begin{assumption}\label{ass:a2}
 There exists a constant $L_2$, such that $|f(0,0)|+|\sigma(0,0)|\leq L_2$.
\end{assumption}
\begin{assumption}\label{ass:a3}
 The function $f$ is locally bounded, continuous and non-decreasing as a function of second variable.
\end{assumption}

The solution for the obstacle problem (\ref{eq:Ell1}) will be a pair $(u,\eta)$ such that $u(x)\geq 0$ on $D$ which satisfies (\ref{eq:Ell1}) in the sense of distribution, and $\eta(dx)$ is a random measure on $D$ which forces the process $u$ to be nonnegative. The rigorous definition of the solution for (\ref{eq:Ell1}) is taken from \cite{WEN} as follows:
\begin{definition}
 A pair $(u,\eta)$ is said to be a solution of problem (\ref{eq:Ell1}) if \\
(i) $u$ is a continuous random field on $D$ satisfying $u(x)\geq 0$ and $u|_{\partial D}$=0 a.s.;\\
(ii) $\eta(dx)$ is a random measure on $D$ such that $\eta(K)<+\infty$ for any compact subset $K\subset D$;\\
(iii) for all $\phi\in C_0^\infty(D)$, the following relation holds
\begin{equation}\label{eq:def1}
 -(u,\Delta\phi)=(f(u),\phi)+\int_D \phi(x)\sigma(x, u(x))W(dx)+\int_D \phi(x)\eta(dx), \quad P-a.s.;
\end{equation}
(iv) $\int_D u(x)\eta(dx)=0$.
\end{definition}
\begin{remark}
(i) We need to explain why we can use the existence and uniqueness results from Theorem 4.1 in \cite{WEN}. In fact, our assumptions \ref{ass:a1}-\ref{ass:a3} are stronger than those mentioned in \cite{WEN}; Besides, the condition (4.2) in \cite{WEN} is covered by our
condition \eqref{eq:ourcondition} in Theorem \ref{thm:main}.

(ii) For better understanding the random measure $\eta$, we briefly discuss the construction of it. We denote
\begin{displaymath}
 \eta_\epsilon(dx):=\frac{1}{\epsilon}(z^\epsilon (x)+v(x))^- dx,
\end{displaymath}
where $z^\epsilon\in L^2(D)\cap C(\overline{D})$ is the unique solution of the following equation
\begin{equation*}
 \left\{
  \begin{aligned}
& -\Delta z^\epsilon (x)=\frac{1}{\epsilon}(z^\epsilon+v)^-(x),\quad x\in D,\\
&\, z^\epsilon|_{\partial D}=0.
\end{aligned}
\right.
\end{equation*}
From the proof of Theorem 2.2 in \cite{NUA}, we know that $\eta_\epsilon$ converges in distributional sense to a distribution $\eta$ on $D$ as $\epsilon$ tends to zero,
i.e. for any $\phi\in C_0^\infty (D)$,
\begin{equation*}
 \lim\limits_{\epsilon\to 0}\int_D \phi(x)\eta_\epsilon(dx)=\int_D \phi(x)\eta(dx).
\end{equation*}
$\eta$ is a positive distribution and hence a measure on $D$.
\end{remark}

\section{The discretization scheme and the main result}\label{sec:main}
For $n\in\mathbb{N}^*$, $h:=1/n$, set
\begin{eqnarray*}
  & &D_n:=\{hi:\ i=(i_1,i_2,\cdots,i_d)\in I_n^d\},\\
  & &\overline{D}_n:=\{hj: j=0,1,\cdots,n\}^d,\quad \partial D_n:=\overline{D}_n\backslash D_n.
\end{eqnarray*}
Let $\delta_j^+$, $\delta_j^-$ be "discrete derivation" operators acting on functions $\psi$ on $D_n$ as follows:
\begin{eqnarray*}
 &&\delta_j^+\psi(x):=n(\psi(x+h e_j)-\psi(x)),\quad \delta_j^-\psi(x):=n(\psi(x)-\psi(x-h e_j)),\\
 &&\Delta_n\psi(x):=\sum_{j=1}^d \delta_j^+\delta_j^-\psi(x)=\sum_{j=1}^dn^2(\psi(x+he_j)-2\psi(x)+\psi(x-he_j)),
\end{eqnarray*}
for $x\in D_n$, where $\{e_j\}_{j=1}^d$ is the standard basis of $\mathbb{R}^d$, and $\psi(x):=0$ for $x\notin D_n$. Define the functions:
\begin{eqnarray*}
 &&k_n(t):=\frac{j}{n},\quad k_n^+(t):=\frac{j+1}{n} \quad \text{for}\quad t\in[\frac{j}{n},\frac{j+1}{n}),\ j=0,\pm 1,\pm 2,\cdots,\\
 &&\underline{k}_n(x):=(k_n(x_1),\cdots,k_n(x_d)),\quad x=(x_1,\cdots,x_d)\in\mathbb{R}^d.
\end{eqnarray*}
Let $(u^n,\eta^n)$ be the solution of the system of reflected stochastic equations:
\begin{equation}\label{eq:Ell2}
 \left\{
  \begin{aligned}
    &-\Delta_n u^n(x)=f(x,u^n(x))+\sigma(x,u^n(x))\delta_1^+\cdots\delta_d^+W(x)+\eta^n(x),\quad x\in D_n;\\[6pt]
    &u^n|_{\partial D_n}=0\,;\quad u^n(x)\geq 0, \quad x\in \overline{D}_n\,;\quad \sum\limits_{x\in D_n} u^n(x)\eta^n(x)=0.
 \end{aligned}
 \right.
\end{equation}
 Now we introduce a ordering method, called the natural ordering (see \cite[page 298]{GRE}):

When we say that the sequence of all $(n-1)^d$ points on the lattice $D_n$ are arranged by a natural ordering, it means that the $k$-th component of this sequence is $\frac{i}{n}$, where $i=(i_1,\cdots,i_d)\in I_n^d$ and satisfying
\begin{equation}\label{eq:simpleranking}
k=i_1+(n-1)(i_2-1)+\cdots+(n-1)^{d-1}(i_d-1).
\end{equation}

Set the sequence $(x_1,\cdots,x_{(n-1)^d})$, where $\{x_i\}_{1\leq i\leq (n-1)^d}$ are all points in $D_n$, is arranged by the simple ranking. Set
\begin{eqnarray*}
 \textbf{u}^n&:=&(u^n(x_1),\cdots,u^n(x_{(n-1)^d})),\\
 \bar{\eta}^n&:=&(\eta^n(x_1),\cdots,\eta^n(x_{(n-1)^d})),\\
 f^n(\textbf{u}^n)&:=&(f(x_1,u^n(x_1)),\cdots,f(x_{(n-1)^d},u^n(x_{(n-1)^d}))),\\
 \sigma^n(\textbf{u}^n)\Delta W^n&:=&(\sigma(x_1,u^n(x_1))\delta_1^+\cdots\delta_d^+W(x_1),\cdots,\\
                   & &\sigma(x_{(n-1)^d},u^n(x_{(n-1)^d})\delta_1^+\cdots\delta_d^+W(x_{(n-1)^d}))).
\end{eqnarray*}

The system (\ref{eq:Ell2}) can be regarded as a $(n-1)^d$-dimensional matrix equation written as
\begin{equation}\label{eq:Ell3}
  -n^2A^n \textbf{u}^n=f^n(\textbf{u}^n)+n^d\sigma^n(\textbf{u}^n) \Delta W^n+\bar{\eta}^n,
\end{equation}
where it is easy to see that $A^n:=A^{n,d}$ is a symmetric linear operator on $\mathbb{R}^{(n-1)^d}$. In particular, when $d=1$, let $A^{n,1}=(A_{ki}^{n,1})$ denote the $(n-1)\times(n-1)$ matrix with elements $A_{kk}^{n,1}=-2$, $A_{ki}^{n,1}=1$ for $|k-i|=1$, $A_{ki}^{n,1}=0$ for $|k-i|>1$. When $d=2,3$, we can also find a $-A^n$ as a positive definite matrix.

 If we set $B^n:=-n^2A^n$ and denote $B^n$ as $B$ for simplicity,
 then
 \begin{equation*}
  \textbf{u}^n=B^{-1}f^n(\textbf{u}^n)+n^d B^{-1}\sigma^n(\textbf{u}^n)\Delta W^n+B^{-1}\bar{\eta}^n.
 \end{equation*}
For $\alpha\in I_n^d$, we define
\begin{equation*}\begin{split}
 & \varphi_\alpha(x)=\varphi_{\alpha_1}(x_1)\cdots\varphi_{\alpha_d}(x_d),\quad \mbox{with}\ \varphi_j(t):=\sqrt{2}\sin(j\pi t), \quad t\in\mathbb{R},\ j=1,2,\cdots;\\
 & \varphi^n_{\alpha_i}(x_i):=\varphi_{\alpha_i}(k_n(x_i))+n[\varphi_{\alpha_i}(k_n^+(x_i))-\varphi_{\alpha_i}(k_n(x_i))](x_i-k_n(x_i)),\quad 1\leq i\leq d;\\
 & \varphi^n_\alpha(x):=\varphi^n_{\alpha_1}(x_1)\cdots\varphi^n_{\alpha_d}(x_d).
\end{split}\end{equation*}
Then we have
 \begin{equation*}
  B^{-1}x=\sum_{\alpha\in I_n^d}\lambda_\alpha^n\langle x,b_\alpha\rangle b_\alpha, \quad x\in\mathbb{R}^{(n-1)^d},
 \end{equation*}
where
\begin{equation*}
 b_\alpha=\left(\frac{1}{n}\right)^{\frac{d}{2}}(\varphi_\alpha(x))_{x\in D_n},\quad \alpha\in I_n^d,
\end{equation*}
 forms an orthonormal basis of $\mathbb{R}^{(n-1)^d}$. Moreover, $b_\alpha$, $\alpha\in I_n^d$ are eigenvectors of $n^2A^n$ with eigenvalues $\lambda_\alpha^n:=\pi^2(\alpha_1^2c_{\alpha_1}+\cdots+\alpha_d^2c_{\alpha_d})$, where $c_j^n:=\sin^2(\frac{j\pi}{2n})(\frac{j\pi}{2n})^{-2}$,
satisfying $\frac{4}{\pi^2}\leq c_j^n\leq 1$.

For $x,y\in[0,1]^d$, denote
\begin{eqnarray}
K_n(x,y)&:=&\sum_{\alpha\in I_n^d}\frac{1}{\lambda_\alpha}\varphi_\alpha(\underline{k}_n(x))\varphi_\alpha(\underline{k}_n(y)),\label{eq:K_ndefinition}\\
K^n(x,y)&:=&\sum\limits_{\alpha\in I_n^d}\frac{1}{\lambda_\alpha}\varphi^n_\alpha(x)\varphi_\alpha(\underline{k}_n(y)),\label{eq:K^ndefinition}\\
K(x,y)&:=&\sum_{\alpha\in I^d} \frac{1}{\pi^2|\alpha|^2}\varphi_\alpha(x)\varphi_\alpha(y),\label{eq:Kdefinition}\\
K'(x,y)&:=&\sum_{\alpha\in I^d} \frac{1}{\pi^2|\alpha|^2}\varphi^n_\alpha(x)\varphi_\alpha(y),\label{eq:K'definition}
\end{eqnarray}
where $K(x,y)$ is the Green function on $D$ associated to the Laplacian operator with Dirichlet boundary conditions, $K_n(x,y)$ is the discretized Green function in finite difference scheme, $K'(x,y)$ and $K^n(x,y)$ are continuous approximations of $K(x,y)$ and $K_n(x,y)$ by linear interpolation, respectively.

The following lemma is concerned with the properties of all four Green functions which can be proved directly or indirectly from \cite[Lemma 3.2, 3.3, 3.4]{GYO}.
\begin{lemma}
 Denote~$G=K, K', K^n, K_n$. There exists a constant $C_1$ depending only on the dimension such that
 \begin{equation}\label{eq:Kestimate}
   \int_D|G(x,y)|^2dy\leq C_1,\quad x\in D.
 \end{equation}
 For any $\epsilon>0$, there exists a constant $B$ such that
 \begin{equation}\label{eq:Klipschitz}
   \int_D|G(x,y)-G(z,y)|^2dy\leq B|x-z|^{4\gamma(d,\epsilon)},\quad x,z\in D,
 \end{equation}
where
\begin{equation}\label{eq:gamma}
\gamma(d,\epsilon):=\begin{cases}
\frac{1}{2}, & \text{if}\ d=1;\\
\frac{1}{2}-\epsilon, &\text{if}\ d=2;\\
\frac{1}{4}-\epsilon, &\text{if}\ d=3.
\end{cases}
\end{equation}
\end{lemma}

Applying \eqref{eq:Ell3} and \eqref{eq:K_ndefinition}, we have
\begin{eqnarray}
 u^n(x)&=&\int_D K_n(x,y)f(\underline{k}_n(y),u^n(\underline{k}_n(y)))dy+\int_D K_n(x,y)\sigma(\underline{k}_n(y),u^n(\underline{k}_n(y)))dW(y)\nonumber\\
 &&+\int_D K_n(x,y)\eta^n(\underline{k}_n(y))dy, \quad x\in D_n.\label{eq:u^n}
\end{eqnarray}

 The following lemma is concerned about the existence and uniqueness of the discretization scheme:
\begin{lemma}\label{le:thm0}
Suppose Assumptions \ref{ass:a1} and \ref{ass:a2} hold with $L_1$ satisfying $\exists p>d/(2\gamma(d,\epsilon)),$
\begin{equation}\label{eq:con_in_lemma}
 2^{2p-1}L_1^p C_D^{p\over 2}+2^{3p-2}c_p L_1^p (aB^{p\over 2}+C_D^{p\over 2})<1,
\end{equation}
where $c_p$ and $a$ are universal constants appeared in the BDG inequality and Komogorov's inequality, $B$ is the constant appeared in the estimate of the Green function $K_n$ in \eqref{eq:Klipschitz}, and $C_D:=\sup\limits_{x\in D}\int_D |K_n(x,y)|^2dy$.
Then there exists a unique solution $(\textbf{u}^n,\bar{\eta}^n)$ to the equation \eqref{eq:Ell3}.
\end{lemma}
\noindent We postpone the proof of Lemma \ref{le:thm0} to Section 5.\\
For $n\in\mathbb{N}^*$, define a continuous approximation $\tilde u^n$ from the values $u^n$ on $D_n$ by linear interpolation.
If $\psi$ is a function on $\{t_i:=i/n: i=0,1,2,\cdots,n\}$, then extending $\psi$ by linear interpolation onto $[0,1]$ means that we define the function $\psi(t):=\psi(t_i)+n(t-t_i)(\psi(t_{i+1})-\psi(t_i))$ for $t\in [t_i,t_{i+1}),\ i=0,1,\cdots,n-1$. Let $\{u^n(x):x\in D_n\}$ defined by \eqref{eq:u^n}. Define $\tilde u^n(x_1,x_2,\cdots,x_d)$ by linear interpolation, successively in each variable $x_1,x_2,\cdots,x_d$. Then
\begin{equation}\label{eq:tildeu}\begin{split}
 \tilde u^n(x)=&\int_D K^n(x,y)f(\underline{k}_n(y),u^n(\underline{k}_n(y)))dy+\int_D K^n(x,y)\sigma(\underline{k}_n(y),u^n(\underline{k}_n(y)))dW(y)\\
 &+\int_D K^n(x,y)\eta^n(\underline{k}_n(y))dy, \quad x\in D.\end{split}
\end{equation}

The following theorem is the main result of this paper.
\begin{theorem}\label{thm:main}
Suppose Assumptions \ref{ass:a1}, \ref{ass:a2} and \ref{ass:a3} hold with $L_1$ satisfying $\exists p>d/(2\gamma(d,\epsilon)),$
\begin{equation}\label{eq:ourcondition}
 2^{3p-2}L_1^p(\tilde C_D)^{\frac{p}{2}}+2^{4p-3}c_pL_1^p(aB^{p\over 2}+\tilde C_D^{p\over 2})<1,
\end{equation}
where $c_p$ and $a$ are universal constants appeared in the BDG inequality and Komogorov's inequality, $B$ is the constant appeared in the estimate of the Green function $K^n$ in \eqref{eq:Klipschitz}, and $\tilde C_D:=\sup\limits_{x\in D}\int_D |K^n(x,y)|^2dy$. Then we have
\begin{equation}\label{thm}
 \lim\limits_{n\to\infty}E[\sup\limits_{x\in D}|\tilde u^n(x)-u(x)|^p]=0.
\end{equation}
\end{theorem}

\section{A discretization scheme for deterministic obstacle problem}\label{sec:dis}
Let $v(x)$ be in $C(D)$ with $v|_{\partial D}=0$.
For $n\in\mathbb{N}^*$, $a=(a_1,\cdots,a_n)$, $b=(b_1,\cdots,b_n)\in\mathbb{R}^n$, we write $a\geq b$ if $a_i\geq b_i$ for all $i=1,\cdots, n$.
Consider a deterministic elliptic PDE with obstacle:
\begin{equation}\label{eq:Det1}
 \begin{cases}
   -\Delta z(x)=\eta(x),\quad& x\in D\,;\\
    z(x)\geq -v(x), \quad& x\in D\,;\\
    z|_{\partial D}=0\,.
 \end{cases}
\end{equation}
The rigorous definition of solution to problem (\ref{eq:Det1}) is as follows:
\begin{definition}\label{def:Det1}
 A pair $(z,\eta)$ is called a solution to (\ref{eq:Det1}) if\\
(i) $z$ is a continuous function on $D$ satisfying $z(x)\geq -v(x)$ and $z|_{\partial D}=0$;\\
(ii) $\eta(dx)$ is a measure on $D$ such that $\eta(K)<\infty$ for all compact subset $K\subset D$;\\
(iii) for all $\phi\in C_0^\infty(D)$, we have
\begin{equation*}
 -(z,\Delta\phi)=\int_D\phi(x)\eta(dx);
\end{equation*}
(iv)$\int_D (z(x)+v(x))\eta(dx)=0.$
\end{definition}
The existence and uniqueness of solution to problem (\ref{eq:Det1}) was proved in \cite[Theorem 2.2]{NUA} and \cite[Theorem 3.1]{WEN}, where the later one is concerned with a more general case with two reflecting walls by means of well-known penalization method.

The following result can be derived from \cite[Theorem 2.2]{NUA}.

\begin{proposition}\label{pro:inequality}
 If $z_1$ and $z_2$ are two solutions to (\ref{eq:Det1}) associated with the functions $v_1$ and $v_2$, respectively, then
\begin{equation*}
 \|z_1-z_2\|_\infty\leq \|v_1-v_2\|_\infty.
\end{equation*}
\end{proposition}

We now introduce the discretization scheme for the deterministic obstacle problem (\ref{eq:Det1}). Choose the sequence $x=(x_1,\cdots,x_{(n-1)^d})$ consisting of all points in $D_n$ arranged by the natural ordering. For $n\in\mathbb{N}^*$, define
\begin{equation*}
 V^n:=(V^n_i)_{1\leq i\leq (n-1)^d}=(v(x_i))_{1\leq i\leq (n-1)^d}.
\end{equation*}

Consider the following linear reflected system in $\mathbb{R}^{(n-1)^d}$:
\begin{equation}\label{eq:Det2}
 \left\{\begin{aligned}
    &-n^2A^nZ^n=\bar{\eta}^n,\\
    &Z^n\geq -V^n,\\
    &\langle Z^n+V^n,\bar{\eta}^n\rangle=0,
  \end{aligned}\right.
\end{equation}
with $\bar{\eta}^n|_{\partial D_n}=0$ and $A^n$ defined in \eqref{eq:Ell3}.
We have
\begin{theorem}\label{thm:uniqueofdeter}
System (\ref{eq:Det2}) admits a unique solution $(Z^n,\bar{\eta}^n)$.
\end{theorem}
\begin{proof}
Since $Z^n\geq-V^n$, then $Z_1^n\geq -V_1^n$, $Z_2^n\geq -V_2^n,\cdots, Z_{n-1}^n\geq -V_{n-1}^n$. Moreover, $\langle Z^n+V^n,\bar{\eta}^n\rangle=\sum_{k=1}^{(n-1)^d}(Z_k^n+V_k^n)\bar{\eta}_k^n=0$, while $\bar{\eta}_k^n\geq 0$.\\
If we set $B=(b_1,\cdots,b_n)':=-n^2A^n$, then $B$ is an invertible matrix and $Z^n=B^{-1}\bar{\eta}^n$.
We deduce, $\bar{\eta}^n_i=-b_iv^n$ if $b_iv^n\leq 0$; $\bar{\eta}^n_i=0$ if $b_iv^n>0$.
Hence, \eqref{eq:Det2} has a unique solution $(Z^n, \bar{\eta}^n)$ when $V^n$ is given.
\end{proof}

\begin{lemma}\label{le:thm1}
 If $(Z^{n_i},\bar{\eta}^{n_i})$ is a solution of (\ref{eq:Det2}) with $V^n$ replaced by $V^{n_i}$, for $i=1,2$, then
\begin{equation*}
 \sup\limits_{1\leq k\leq {(n-1)^d}}|Z_k^{n_1}-Z_k^{n_2}|\leq \sup\limits_{1\leq k\leq {(n-1)^d}}|V_k^{n_1}-V_k^{n_2}|.
\end{equation*}
\end{lemma}
In order to prove Lemma \ref {le:thm1}, we need the following result (see \cite[Lemma 3.1]{ZHA}).
\begin{lemma}\label{le:thm2}
 $\langle b^+,A^nb\rangle\leq 0$ for all $b\in \mathbb{R}^{(n-1)^d}$.
\end{lemma}
\begin{proof}[\textit{Proof of Lemma \ref{le:thm1} :}]
Set $m:=\sup\limits_k|V_k^{n_1}-V_k^{n_2}|$ and $M:=(m,\cdots,m)\in\mathbb{R}^{(n-1)^d}$. As $(Z^{n_i},\bar{\eta}^{n_i})$ is a solution of (\ref{eq:Det2}), we have
\begin{equation*}
-n^2A^n(Z^{n_1}-Z^{n_2}-M)-n^2A^nM=\bar{\eta}^{n_1}-\bar{\eta}^{n_2}.
\end{equation*}
Multiplying the above equality by $(Z^{n_1}-Z^{n_2}-M)^+$ we get
\begin{small}\begin{eqnarray*}
  &&\langle-n^2A^n(Z^{n_1}-Z^{n_2}-M),(Z^{n_1}-Z^{n_2}-M)^+\rangle\\
  &=&\langle n^2A^nM,(Z^{n_1}-Z^{n_2}-M)^+\rangle+\langle\bar{\eta}^{n_1},(Z^{n_1}-Z^{n_2}-M)^+\rangle-\langle\bar{\eta}^{n_2},(Z^{n_1}-Z^{n_2}-M)^+\rangle.
\end{eqnarray*}\end{small}
Noting the definition of $A^n$, which is a diagonally dominant matrix. For example, when $d=1$, $A^{n,1}M=(-m,0,\cdots,0,-m)$, then
\begin{equation*}
\langle n^2A^{n,1}M,(Z^{n_1}-Z^{n_2}-M)^+\rangle=-n^2m(Z_1^{n_1}-Z_1^{n_2}-m)^+-n^2m(Z_{n-1}^{n_1}-Z_{n-1}^{n_2}-m)^+\leq 0.
\end{equation*}
When $d=2,3$, letting $A^n=A^{n,d}$, we also have
\begin{equation*}
  \langle n^2A^{n}M,(Z^{n_1}-Z^{n_2}-M)^+\rangle\leq 0.
\end{equation*}
Observe that
\begin{equation*}
\{k; Z_k^{n_1}-Z_k^{n_2}>m\}\subset\{k; Z_k^{n_1}>-V_k^{n_2}+m\}\subset\{k; Z_k^{n_1}>-V_k^{n_1}\}.
\end{equation*}
This yields
\begin{equation*}
\langle\bar{\eta}^{n_1},(Z^{n_1}-Z^{n_2}-M)^+\rangle \leq\sum_{k=1}^{n-1}(Z_k^{n_1}-Z_k^{n_2}-m)^+1_{\{Z_k^{n_1}>-V_k^{n_1}\}}\bar{\eta}_k^{n_1}=0.
\end{equation*}
Moreover, it is obvious that $\langle\bar{\eta}^{n_2},(Z^{n_1}-Z^{n_2}-M)^+\rangle\geq0$.
Therefore, $\langle-n^2A^n(Z^{n_1}-Z^{n_2}-M),(Z^{n_1}-Z^{n_2}-M)^+\rangle\leq 0$.
By Lemma \ref{le:thm2}, we deduce
\begin{equation*}
 |(Z^{n_1}-Z^{n_2}-M)^+|^2= 0,
\end{equation*}
which means
$Z^{n_1}-Z^{n_2}\leq M.$
\end{proof}
Choose the sequence $y=(y_1,\cdots,y_{(n-1)^d})$ consisting of all points in $D_n$ arranged by the natural ordering.
For $n\in\mathbb{N}^*$, define the continuous function $z^n$ by choosing $z^n(y_k)=Z^n_k, 1\leq k\leq (n-1)^d$ and from the values $z^n$ by linear interpolation in each variable $x_1,\cdots,x_d$ successively. Indeed, when $d=3$, $(x_1,x_2,x_3)\in [\frac{k_1}{n},\frac{k_1+1}{n})\times[\frac{k_2}{n},\frac{k_2+1}{n})\times[\frac{k_3}{n},\frac{k_3+1}{n})$, we have
\begin{eqnarray*}
&&z^n(x_1,\frac{k_2}{n},\frac{k_3}{n})=z^n(\frac{k_1}{n},\frac{k_2}{n},\frac{k_3}{n})+n(x_1-\frac{k_1}{n})(z^n(\frac{k_1+1}{n},\frac{k_2}{n},\frac{k_3}{n})-z^n(\frac{k_1}{n},\frac{k_2}{n},\frac{k_3}{n})),\\
&&z^n(x_1,x_2,\frac{k_3}{n})=z^n(x_1,\frac{k_2}{n},\frac{k_3}{n})+n(x_2-\frac{k_2}{n})(z^n(x_1,\frac{k_2+1}{n},\frac{k_3}{n})-z^n(x_1,\frac{k_2}{n},\frac{k_3}{n})),\\
&&z^n(x_1,x_2,x_3)=z^n(x_1,x_2,\frac{k_3}{n})+n(x_3-\frac{k_3}{n})(z^n(x_1,x_2,\frac{k_3+1}{n})-z^n(x_1,x_2,\frac{k_3}{n})).
\end{eqnarray*}
Similarly, define $\eta^n(x),\ x\in D$ by choosing $\eta^n(y_k)=\bar{\eta}^n_k, 1\leq k\leq (n-1)^d$ and linear interpolation from the values $\eta^n$ on $D_n$.

\begin{theorem}\label{thm:conz}
 Let $z$ be the solution of (\ref{eq:Det1}). Then
\begin{equation*}
 \lim\limits_{n\to\infty}\sup\limits_{x\in D}|z^n(x)-z(x)|=0.
\end{equation*}
\end{theorem}
\begin{proof}
We divide the proof into two steps.

(I) Suppose $v\in C^2(D)$. We have
\begin{equation*}
Z^n=B^{-1}\bar{\eta}^n.
\end{equation*}
Recall the kernel $K^n(x,y)$ by \eqref{eq:K^ndefinition}.
It is easy to verify that
\begin{equation*}
 z^n(x)=\int_D K^n(x,y)\eta^n(\underline{k}_n(y))dy.
\end{equation*}

Next, we estimate $\eta^n$. In view of $V^n|_{\partial D_n}=0$, for $k=1,\cdots,(n-1)^d$, either $\bar{\eta}_k^n=0$ or $\bar{\eta}_k^n>0$ with $Z_k^n=-V^n_k$, then we have for $x\in D_n$
\begin{equation}\label{eq:eta}\begin{split}
  |\eta^n(x)|&\leq \sum\limits_{j=1}^d n^2|v(x+\frac{1}{n}e_j)-2v(x)+v(x-\frac{1}{n}e_j)|\\
  &=\sum\limits_{j=1}^d n^2|\int_{x_j}^{x_j+\frac{1}{n}}dy\int_{x_j}^y \frac{\partial^2 v(z)}{\partial z_j^2}dz_j+\int_{x_j-\frac{1}{n}}^{x_j}dy\int^{x_j}_y \frac{\partial^2 v(z)}{\partial z_j^2}dz_j|\\&\leq 2d\|v\|_2.
\end{split}\end{equation}
Therefore, by \eqref{eq:eta} and the smoothness assumption on $v$, we have
\begin{equation}\label{eq:eta^nbounded}
 \int_D |\eta^n(\underline{k}_n(y))|^2dy\leq\sup\limits_{x\in D_n}|\eta^n(x)|\leq C.
\end{equation}
In view of \eqref{eq:Klipschitz} and \eqref{eq:eta^nbounded}, by H\"{o}lder's inequality, we deduce that there exists a constant $C$ depending only on $(d,\epsilon),\ \forall \epsilon >0$, such that
\begin{equation}\label{eq:con1}\begin{split}
  |z^n(x)-z^n(z)|^2&\leq |\int_D(K^n(x,y)-K^n(z,y))\eta^n(\underline{k}^n(y))dy|^2\\
  &\leq (\int_D|K^n(x,y)-K^n(z,y)|^2dy)(\int_D|\eta^n(\underline{k}_n(y))|^2dy)\\
  &\leq C|x-z|^{4\gamma(d,\epsilon)}, \quad x,z\in D.
\end{split}\end{equation}

By Arzela-Ascoli theorem, we know that $\{z^n(x), n\geq 1\}$ is relatively compact. On the other hand, $\{\eta^n, n\geq 1\}$ is relatively compact in $L^2(D)$ with respect to the weak topology. Selecting a subsequence if necessary, we can assume that $z^n(\cdot)$ converges uniformly to some function $z(\cdot)\in C(D)$ and $\eta^n(\underline{k}_n(\cdot))$ converges weakly to some $\eta(\cdot)\in L^2(D)$. We complete the first step by showing that $(z,\eta)$ is the solution of system (\ref{eq:Det1}). Choose the sequence $x=(x_1,\cdots,x_{(n-1)^d})$ consisting of all points in $D_n$ arranged by the natural ordering. For $\phi\in C_0^\infty(D)$, set $\phi^n:=(\phi(x_1),\cdots,\phi(x_{(n-1)^d}))$. By the symmetry of $A^n$, we have
\begin{equation*}
 -\langle n^2A^n\phi^n,Z^n\rangle =\langle \phi^n,\bar{\eta}^n\rangle.
\end{equation*}
Multiplying the above equation by $\frac{1}{n^d}$ we get
\begin{equation*}
 -\int_D \Delta_n\phi(\underline{k}_n(y))z^n(\underline{k}_n(y))dy=\int_D\phi(\underline{k}_n(y))\eta^n(\underline{k}_n(y))dy.
\end{equation*}
 Letting $n\to\infty$, in view of the strong convergence of $\phi(\underline{k}_n(y))$ we obtain
\begin{equation*}
 -\int_D\Delta\phi(y)z(y)dy=\int_D\phi(y)\eta(y)dy.
\end{equation*}
On the other hand, it follows from the definition that
\begin{equation}\label{eq:thm1}
 \int_D(z^n(\underline{k}_n(y))+v(\underline{k}_n(y)))\eta^n(\underline{k}_n(y))dy=0.
\end{equation}
Invoking (\ref{eq:con1}) and the dominated convergence theorem we have
\begin{equation}\label{eq:thm2}\begin{split}
 &\int_D(z^n(\underline{k}_n(y))+v(\underline{k}_n(y))-z(y)-v(y))^2dy\\
\leq&\, C\int_D(z^n(\underline{k}_n(y))-z^n(y))^2dy+C\int_D(z^n(y)-z(y))^2dy+C\int_D(v(\underline{k}_n(y))-v(y))^2dy\\
\leq&\, C(\frac{1}{n})^{4\gamma(d,\epsilon)}+C\int_D(z^n(y)-z(y))^2dy+C\int_D(v(\underline{k}_n(y))-v(y))^2dy\to 0, \quad n\to\infty.
\end{split}
\end{equation}
Letting $n\to\infty$ in (\ref{eq:thm1}), in view of the weak convergence of $\eta^n$ and (\ref{eq:thm2}) we have
\begin{equation*}
 \int_D(z(y)+v(y))\eta(y)dy=0.
\end{equation*}
Thus, $(z,\eta)$ is a solution of (\ref{eq:Det1}).

(II) For the general case $v\in C(D)$, we take a sequence $v^m\in C^2(D), m\in\mathbb{N}^*$ such that $\sup_{x\in D}|v^m(x)-v(x)|\to 0$ as $m\to \infty$. Choose the sequence $x=(x_1,\cdots,x_{(n-1)^d})$ consisting of all points in $D_n$ arranged by the natural ordering. For $n\in\mathbb{N}^*$, we define
\begin{equation*}
 V^{m,n}:=(v^m(x_i))_{1\leq i\leq (n-1)^d}.
\end{equation*}
Let $(Z^{m,n},\eta^{m,n})$ be the solution to the following obstacle problem in $\mathbb{R}^{(n-1)^d}$:
\begin{equation*}
\left\{ \begin{aligned}
    &-n^2A^nZ^{m,n}=\bar{\eta}^{m,n};\\
    &Z^{m,n}\geq -V^{m,n};\\
    &\langle Z^{m,n}+V^{m,n},\bar{\eta}^{m,n}\rangle=0.
  \end{aligned}\right.
\end{equation*}
Mimic the definition of $z^n$. Introduce the continuous functions $z^{m,n}$ by setting $z^{m,n}(x_k)=Z^{m,n}_k$ for $1\leq k\leq (n-1)^d$ and from the values of $z^{m,n}$ on $D_n$ by linear interpolation.

 According to the result proved in the first step, for $m\in\mathbb{N}^*$, we have
\begin{equation}\label{eq:thm3}
 \lim\limits_{n\to\infty}\sup\limits_{x\in D}|z^{m,n}(x)-z^{(m)}(x)|=0,
\end{equation}
where $z^{(m)}(x)$ is the solution of  the following elliptic obstacle problem:
\begin{equation*}
 \left\{\begin{aligned}
    &\frac{\partial^2 z^{(m)}(x)}{\partial x^2}=\eta^{(m)}(x);\\
    &z^{(m)}(x)\geq -v^m(x);\\
    &\int_D(z^{(m)}(x)+v^m(x))\eta^{(m)}(dx)=0.
  \end{aligned}\right.
\end{equation*}
By the definition of $z^{m,n}$ and $z^n$, the difference of $z^{m,n}$ and $z^n$ at point $x\in D$ can be decided by at most $2^d$ points in the neighbor of $x$ on the lattice $D_n$. For example, when $d=1$, and $x\in[\frac{k}{n},\frac{k+1}{n})$, we have
\begin{equation*}\begin{split}
  |z^{m,n}(x)-z^n(x)|&=|(Z_k^{m,n}-Z_k^n)+(x-\frac{k}{n})\frac{(Z_{k+1}^{m,n}-Z_{k+1}^n)-(Z_k^{m,n}-Z_k^n)}{\frac{k+1}{n}-\frac{k}{n}}|\\
  &\leq|(Z_k^{m,n}-Z_k^n)\vee (Z_{k+1}^{m,n}-Z_{k+1}^n)|.
\end{split}\end{equation*}
Therefore,
\begin{equation}
 \sup\limits_{x\in D}|z^{m,n}(x)-z^n(x)|=\sup\limits_{1\leq k\leq (n-1)^d}|Z_k^{m,n}-Z_k^n|.
\end{equation}
Then, in view of Lemma \ref{le:thm1}, we have
\begin{equation}\label{eq:thm4}\begin{split}
  \sup\limits_{x\in D}|z^{m,n}(x)-z^n(x)|
&=\sup\limits_{1\leq k\leq (n-1)^d}|Z_k^{m,n}-Z_k^n|
\leq\sup\limits_{1\leq k\leq (n-1)^d}|V_k^{m,n}-V_k^n|\\
&=\sup\limits_{x\in D_n}|v^m(x)-v(x)|
\leq\sup\limits_{x\in D}|v^m(x)-v(x)|.
\end{split}\end{equation}
Besides, in view of Proposition \ref{pro:inequality}, we get
\begin{equation}\label{eq:thm5}\begin{split}
 \sup\limits_{x\in D}|z^n(x)-z(x)|
 \leq& \sup\limits_{x\in D}|z^n(x)-z^{(m)}(x)|+\sup\limits_{x\in D}|z^{(m)}(x)-z(x)|\\
 \leq &\sup\limits_{x\in D}|z^n(x)-z^{m,n}(x)|+\sup\limits_{x\in D}|z^{m,n}(x)-z^{(m)}(x)|+\sup\limits_{x\in D}|v^m(x)-v(x)|\\
 \leq & 2\sup\limits_{x\in D}|v^m(x)-v(x)|+\sup\limits_{x\in D}|z^{m,n}(x)-z^{(m)}(x)|.
\end{split}\end{equation}
For $\epsilon>0$, one can choose $m$ sufficiently large such that
\begin{equation}\label{eq:thm6}
 2\sup\limits_{x\in D}|v^m(x)-v(x)|\leq \frac{\epsilon}{2}.
\end{equation}
For such a fixed $m$, we deduce from (\ref{eq:thm3}) that there exists an integer $N$ such that, for $n\geq N$,
\begin{equation}\label{eq:thm7}
 \sup\limits_{x\in D}|z^{m,n}(x)-z^{(m)}(x)|\leq \frac{\epsilon}{2}.
\end{equation}
Combing (\ref{eq:thm5}), (\ref{eq:thm6}) and (\ref{eq:thm7}), we obtain that
\begin{equation*}
 \sup\limits_{x\in D}|z^n(x)-z(x)|\leq\epsilon,
\end{equation*}
for $n\geq N$. As $\epsilon$ is arbitrary, the proof is complete.
\end{proof}

\section{The convergence of the scheme}
Before we can complete the proof of Lemma \ref{le:thm0}, we introduce a priori estimate for the stochastic integrals. The following lemma can be proved by Kolmogrov's lemma and BDG inequality (see \cite[page 159-160]{WEN} for details).
\begin{lemma}\label{le:kolmogorov}
Let $\psi$ be a continuous random field on $D$, $\hat G(x,y)$ is a continuous function satisfying
 $$\|\hat G(x,y)-\hat G(x',y)\|^2_{L^2(D)}\leq B|x-x'|^{4\gamma(d,\epsilon)},$$
 $$\hat C_D:=\sup\limits_{x\in D}\int_{D}|\hat G(x,y)|^2dy<\infty,$$
where $\gamma(d,\epsilon)$ is defined by \eqref{eq:gamma}. Set
$$I(x):=\int_D \hat G(x,y)\psi(y)dW(y).$$
Then for any $p>d/(2\gamma(d,\epsilon))$ we have
$$E[\sup\limits_{x\in D}|I(x)|^p]\leq 2^{p-1}c_p(aB^{p\over 2}+\hat C_D^{p\over 2})E[\|\psi\|_\infty^p],$$
where $c_p$ and $a$ are constants appeared in the BDG inequality and Komogorov's inequality.
\end{lemma}

\begin{proof}[\textit{Proof of Lemma \ref{le:thm0}:}]
 We consider the following iteration:
 for fixed $n\in\mathbb{N}^*$, set $\mathbf{u}^{n,0}=0$ and
 \begin{equation}\label{eq:sec5ell0}
\left\{ \begin{aligned}
   &-n^2A^nV^{n,1}=f^n(0)+n^d\sigma^n(0)\Delta W^n;\\
   &-n^2A^nZ^{n,1}=\bar{\eta}^{n,1};\\
   &Z^{n,1}\geq-V^{n,1};\\
   &\langle Z^{n,1}+V^{n,1},\bar{\eta}^{n,1}\rangle=0.
  \end{aligned}\right.
 \end{equation}
Since $-n^2A^n$ is reversible, the stochastic linear equation admits a unique solution $V^{n,1}$; From the result for the deterministic obstacle problem (see Theorem \ref{thm:uniqueofdeter}), we know that for almost all $\omega$, $(Z^{n,1},\bar{\eta}^{n,1})$ uniquely exists. Hence, $(\mathbf{u}^{n,1}:=Z^{n,1}+V^{n,1},\bar{\eta}^{n,1})$ is the unique solution of \eqref{eq:sec5ell0}.

We iterate this procedure. Suppose, for $m\geq 2$, $\mathbf{u}^{n,m-1}$ has been defined, and $\mathbf{u}^{n,m}$ satisfies the following equation:
 \begin{equation*}
\left\{ \begin{aligned}
   &-n^2A^nV^{n,m}=f^n(\mathbf{u}^{n,m-1})+n^d\sigma^n(\mathbf{u}^{n,m-1})\Delta W^n;\\
   &-n^2A^nZ^{n,m}=\bar{\eta}^{n,m};\\
   &Z^{n,m}\geq-V^{n,m};\\
   &\langle Z^{n,m}+V^{n,m},\bar{\eta}^{n,m}\rangle=0.
  \end{aligned}\right.
 \end{equation*}
Then, $(\mathbf{u}^{n,m}:=Z^{n,m}+V^{n,m},\bar{\eta}^{n,m})$ exists and is unique.

By Lemma \ref{le:thm1}, we have
 \begin{equation*}\begin{split}
  \|\mathbf{u}^{n,m}-\mathbf{u}^{n,m-1}\|_{\infty} &= \sup\limits_{1\leq k\leq (n-1)^d} |\mathbf{u}^{n,m}_k-\mathbf{u}^{n,m-1}_k|\\
  &\leq \sup\limits_{1\leq k\leq (n-1)^d} |(V^{n,m}_k-V^{n,m-1}_k)+(Z^{n,m}_k-Z^{n,m-1}_k)|\\
  &\leq\, 2\sup\limits_{1\leq k\leq (n-1)^d}|(V^{n,m}_k-V^{n,m-1}_k)|.
\end{split}\end{equation*}
This yields
\begin{equation*}
 E\sup\limits_{1\leq k\leq (n-1)^d}|\mathbf{u}^{n,m}_k-\mathbf{u}^{n,m-1}_k|^p\leq 2^p E\sup\limits_{1\leq k\leq (n-1)^d}|V^{n,m}_k-V^{n,m-1}_k|^p.
\end{equation*}
 Choose the sequence $x=(x_1,\cdots,x_{(n-1)^d})$ consisting of all points in $D_n$ arranged by the natural ordering. Define a discrete random field $u^{n,r}(y),\ r\geq 0,\ y\in D_n$, such that $u^{n,r}(x_i)=\mathbf{u}^{n,r}_i,\ 1\leq i\leq (n-1)^d$. Define a discrete random field $v^{n,r}(y),\ r\geq 0,\ y\in D_n$, such that $v^{n,r}(x_i)=V^{n,r}_i,\ 1\leq i\leq (n-1)^d$. For $x\in D_n$
\begin{equation*}\begin{split}
 v^{n,m}(x)=&\int_D K_n(x,y)f(\underline{k}_n(y),u^{n,m-1}(\underline{k}_n(y)))dy\\
 &+\int_D K_n(x,y)\sigma(\underline{k}_n(y),u^{n,m-1}(\underline{k}_n(y)))dW(y).\end{split}
\end{equation*}
Denote
$$I(k):=V^{n,m}_k-V^{n,m-1}_k=v^{n,m}(x_k)-v^{n,m-1}(x_k),\quad 1\leq k\leq (n-1)^d.$$
$$I_1(x_k):=\int_D K_n(x_k,y)[f(\underline{k}_n(y),u^{n,m-1}(\underline{k}_n(y)))-f(\underline{k}_n(y),u^{n,m-2}(\underline{k}_n(y)))]dy,$$
$$I_2(x_k):=\int_D K_n(x_k,y)[\sigma(\underline{k}_n(y),u^{n,m-1}(\underline{k}_n(y)))-\sigma(\underline{k}_n(y),u^{n,m-2}(\underline{k}_n(y)))]dW(y).$$
Then we have
\begin{equation*}
I(k)=I_1(x_k)+I_2(x_k).
\end{equation*}

Through simple calculations, in view of Assumption \ref{ass:a1}, we have
\begin{equation}\label{eq:estin1_in_cha3}
  E[\sup\limits_{1\leq k\leq (n-1)^d}|I_1(x_k)|^p]\leq L_1^p C_D^{p\over 2}E[\sup\limits_{y\in D}|(u^{n,m-1}-u^{n,m-2})(\underline{k}_n(y))|^p].
\end{equation}
In view of Lemma \ref{le:kolmogorov}, formula \eqref{eq:Klipschitz} and Assumption \ref{ass:a1}, we obtain
\begin{equation}\label{eq:estin2_in_cha3}
 E[\sup\limits_{1\leq k\leq (n-1)^d}|I_2(x_k)|^p]\leq 2^{p-1}c_p L_1^p (aB^{p\over 2}+C_D^{p\over 2})E[\sup\limits_{y\in D}|(u^{n,m-1}-u^{n,m-2})(\underline{k}_n(y))|^p].
\end{equation}
In view of formulas \eqref{eq:estin1_in_cha3} and \eqref{eq:estin2_in_cha3}, we have
\begin{eqnarray*}
  E[\sup\limits_{1\leq k\leq (n-1)^d} |I(k)|^p]&\leq& 2^{p-1}E[\sup\limits_{1\leq k\leq (n-1)^d}|I_1(x_k)|^p]+2^{p-1}E[\sup\limits_{1\leq k\leq (n-1)^d}|I_2(x_k)|^p]\\
  &\leq&
  [2^{p-1}L_1^p C_D^{p\over 2}+2^{2p-2}c_p L_1^p (aB^{p\over 2}+C_D^{p\over 2})]\\
  &&\times E[\sup\limits_{y\in D} |(u^{n,m-1}-u^{n,m-2})(\underline{k}_n(y))|^p].
\end{eqnarray*}
Let $\tilde{C}:= 2^{2p-1}L_1^p C_D^{p\over 2}+2^{3p-2}c_p L_1^p (aB^{p\over 2}+C_D^{p\over 2})$, we have
\begin{equation*}\begin{split}
E[\sup\limits_k |\mathbf{u}^{n,m}_k-\mathbf{u}^{n,m-1}_k|^p]\leq &\, \tilde{C} E[\sup\limits_k|\mathbf{u}^{n,m-1}_k-\mathbf{u}^{n,m-2}_k|^p]\\
\leq&\,\cdots\leq \,\tilde{C}^{m-1} E[\sup\limits_k |\mathbf{u}^{n,1}_k-\mathbf{u}^{n,0}_k|^p].
\end{split}\end{equation*}
In view of the condition \eqref{eq:con_in_lemma}, there exists a $p$ such that $\tilde C<1$. Then we obtain that for any $m_1\geq m_2\geq 1$,
 \begin{equation*}
  E\|\mathbf{u}^{n,m_1}-\mathbf{u}^{n,m_2}\|_\infty^p\to 0, \quad \text{as}\ m_1,\ m_2\to \infty.
 \end{equation*}
 Hence, there exists a $\mathbf{u}^n$ such that
 \begin{equation*}
  \lim\limits_{m\to\infty} E\sup\limits_k |\mathbf{u}^{n,m}_k-\mathbf{u}^n_k|^p=0.
 \end{equation*}

Similarly, we can prove the existence of $V^n$ satisfying
\begin{equation*}
 \lim\limits_{m\to\infty} E\sup\limits_k |V^{n,m}_k-V^n_k|^p=0.
\end{equation*}
Hence, by Lemma \ref{le:thm1}, we deduce there exists a random vector $Z^n$ such that
\begin{equation*}
 \lim\limits_{m\to\infty} E\sup\limits_k |Z^{n,m}_k-Z^n_k|^p=0.
\end{equation*}
Then, as $\bar{\eta}^{n,m}$ is a $(n-1)^d$-dim vector, when $m\to\infty$, we have $\bar{\eta}^{n,m}\to\bar{\eta}^n$ due to $Z^{n,m}\to Z^n$. Moreover, we have $\langle Z^n+V^n,\bar{\eta}^n\rangle=0$. Put $\mathbf{u}^n=Z^n+V^n$, then $(\mathbf{u}^n,\bar{\eta}^n)$ is a solution of (\ref{eq:Ell3}).\\
Now, we come to the uniqueness.
Assume that $(\mathbf{u}_1^n,\bar{\eta}_1^n)$ and $(\mathbf{u}_2^n,\bar{\eta}_2^n)$ are two solutions of (\ref{eq:Ell3}). From the above calculation, we know
\begin{equation*}
 E[\|\mathbf{u}_1^n-\mathbf{u}_2^n\|_\infty^p] \leq 2^p E[\|V_1^n-V_2^n\|^p_\infty]\leq \tilde C E [|\mathbf{u}_1^n-\mathbf{u}_2^n|^p].
\end{equation*}
Since $\tilde C<1$ for some $p$, so we have $\mathbf{u}_1^n=\mathbf{u}_2^n, \quad a.s.$\\
On the other hand, choose the sequence $x=(x_1,\cdots,x_{(n-1)^d})$ consisting of all points in $D_n$ arranged by the natural ordering. We have for any $\phi^n\in C^\infty(D)$
\begin{equation*}
 \sum\limits_{k=1}^{(n-1)^d}\phi^n(x_k)((\bar{\eta}_1^n)_k-(\bar{\eta}_2^n)_k)=0,
\end{equation*}
which means
$ \bar{\eta}_1^n=\bar{\eta}_2^n, \ a.s..$

\end{proof}

Let $V^n$ be the solution of the following stochastic equation:
\begin{equation}
 -n^2A^nV^n=f^n(\mathbf{u}^n)+n^d \sigma^n(\mathbf{u}^n) \Delta W^n,
\end{equation}
and  $(\mathbf{u}^n,\bar{\eta}^n)$ solve (\ref{eq:Ell3}), then $(Z^n:= \mathbf{u}^n-V^n, \bar{\eta}^n)$ is the solution of the following problem:
\begin{equation}\label{eq:sec5ell1}
 \left\{\begin{aligned}
    &-n^2A^nZ^n=\bar{\eta}^n;\\
    &Z^n\geq -V^n;\\
    &\langle Z^n+V^n,\bar{\eta}^n\rangle =0.
  \end{aligned}\right.
\end{equation}
Recall the continuous random field $\tilde u^n(x)$ defined in (\ref{eq:tildeu}). Choose the sequence\\
 $z=(z_1,\cdots,z_{(n-1)^d})$ consisting of all points in $D_n$ arranged by the natural ordering. Define the continuous random fields $\eta^n(x),\ v^n(x)$ by setting $\eta^n(z_k)=\bar{\eta}^n_k$ and $v^n(z_k)=V^n_k$ for $1\leq k\leq (n-1)^d$, and by setting $\eta^n(x)$ and $v^n(x)$ for $x\in D\backslash D_n$ from the values $\eta^n$ and $v^n$ on $D_n$ respectively by linear interpolation,
with $\eta^n|_{\partial D}=0$ and $v^n|_{\partial D}=0$. Let the kernel $K^n(x,y)$ be defined as in \eqref{eq:K^ndefinition}. It is easy to verify that $v^n$ satisfies the following relation:
\begin{equation}\label{eq:v^n1}
  v^n(x)=\int_D K^n(x,y)f(\underline{k}_n(y),u^n(\underline{k}_n(y)))dy+\int_D K^n(x,y)\sigma(\underline{k}_n(y),u^n(\underline{k}_n(y)))W(dy).
\end{equation}

Now we come to prove our main theorem.

\begin{proof}[{\bf Proof of Theorem \ref{thm:main}:}]
 We know that \eqref{eq:def1} can be written as the integral formulation (see \cite[Lemma 2.3]{BUC4}). Recalling the definition of the Green function $K(x,y)$ we have
\begin{equation*}
 u(x)=\int_D K(x,y)f(y,u(y))dy+\int_DK(x,y)\sigma(y,u(y))W(dy)+\int_D K(x,y)\eta(dy).
\end{equation*}
Set
\begin{equation*}
 \bar{v}(x):=\int_D K(x,y)f(y,u(y))dy+\int_D K(x,y)\sigma(y,u(y))W(dy).
\end{equation*}
Then $(\bar{z}(x):=u(x)-\bar{v}(x),\eta(x))$ solves the following random elliptic obstacle problem:
\begin{equation}
\left\{ \begin{aligned}
    &-\Delta\bar{z}(x)=\eta(x);\\
    &\bar{z}(x)\geq -\bar{v}(x),\quad x\in D;\\
    &\int_D(\bar{z}(x)+\bar{v}(x))\eta(dx)=0.
  \end{aligned}\right.
\end{equation}
Choose the sequence $y=(y_1,\cdots,y_{(n-1)^d})$ consisting of all points in $D_n$ arranged by the natural ordering. For $n\in\mathbb{N}^*$, define
\begin{equation*}
 \bar{V}^n:=(\bar{V}^n_i)_{1\leq i\leq (n-1)^d}=(\bar{v}(y_i))_{1\leq i\leq (n-1)^d}.
\end{equation*}

We consider the following random obstacle problem in $\mathbb{R}^{(n-1)^d}$:
\begin{equation}\label{eq:sec5ell2}
\left\{ \begin{aligned}
    &-n^2A^n\bar{Z}^n=\bar{\eta}^n;\\
    &\bar{Z}^n\geq -\bar{V}^n;\\
    &\langle\bar{Z}^n+\bar{V}^n,\bar{\eta}^n\rangle=0.
\end{aligned}\right.
\end{equation}
Introduce the continuous random field $\bar{z}^n(x)$ by setting $\bar{z}^n(y_k)=\bar{Z}^n_k$ for $1\leq k\leq (n-1)^d$ and for $x\in D$ by linear interpolation from the values $\bar{z}^n$ on $D_n$ and $\bar{z}^n|_{\partial D_n}=0$. By Theorem \ref{thm:conz}, we have
\begin{equation}\label{eq:sec5con1}
 \lim\limits_{n\to\infty}\sup\limits_{x\in D}|\bar{z}^n(x)-\bar{z}(x)|=0.
\end{equation}
Introduce the continuous random field $\bar{v}^n(x)$ by setting $\bar{v}^n(y_k)=\bar{V}^n_k$ for $1\leq k\leq (n-1)^d$ and for $x\in D$ by linear interpolation from the values $\bar{v}^n$ on $D_n$ and $\bar{v}^n|_{\partial D_n}=0$.
 Since $E(\|u\|_\infty^2)<\infty$ (by \cite[Theorem 4.1]{WEN}) and $|\varphi^n_\alpha(x)-\varphi_\alpha(x)|\leq 2^d|\alpha|/n$ (by \cite[Lemma 3.2]{GYO}), it is clear that for any $p\geq 1$,
\begin{equation}\label{eq:sec5con2}
 \lim\limits_{n\to\infty}E[\sup\limits_{x\in D}|\bar{v}^n(x)-\bar{v}(x)|^p]=0.
\end{equation}

In fact, in view of \cite[Theorem 2.2]{ZHA1}, to prove \eqref{eq:sec5con2}, it is sufficient to prove the following two assertions: there exist $p\geq 1$ and $C>0$, such that for any $x,y\in D$ we have

\vspace{2mm}
 (i) $E[|\bar{v}^n(x)-\bar{v}^n(y)|^p+|\bar{v}(x)-\bar{v}(y)|^p]\leq C|x-y|^{2p\gamma(d,\epsilon)};$

\vspace{2mm}
 (ii) $\lim\limits_{n\to\infty}E[|\bar{v}^n(x)-\bar{v}(x)|^p]=0.$

\vspace{2mm}
 In view of formula \eqref{eq:Klipschitz}, Assumptions \ref{ass:a1} and \ref{ass:a2}, for any $x,z\in D$, we have
 \begin{eqnarray*}
  &&E[|\bar{v}(x)-\bar{v}(z)|^p]\\
  &\leq& 2^{p-1}(1+c_p)C(L_1,L_2)E[(1+\|u\|^2_\infty)^{p\over 2}](\int_D |K(x,y)-K(z,y)|^2dy)^{p\over 2}\\
  &\leq& 2^{p-1}(1+c_p)C(L_1,L_2)E[(1+\|u\|^2_\infty)^{p\over 2}]B^{p\over 2}|x-z|^{2p\gamma(d,\epsilon)}.
 \end{eqnarray*}
 Replace $\bar{v}$ in the above estimate by $\bar{v}^n$, the assertion still holds. The assertion (i) is proved.

 On the other hand, by \cite[Lemma 3.2]{GYO}, we know
\begin{equation}\label{eq:K'Kdifference}
 \sup\limits_{x\in D}\int_D(K'(x,y)-K(x,y))^2dy\leq C(\frac{1}{n})^{4\gamma(d,\epsilon)}.
\end{equation}
 Hence, we have
 \begin{equation*}\begin{split}
 &E[|\bar{v}^n(x)-\bar{v}(x)|^p]\\
 \leq&\, C_p E\left[\Big(\int_D|K'(x,y)-K(x,y)|^2f^2(y,u(y))dy\Big)^{\frac{p}{2}}\right]\\
 &+c_p E\left[\Big(\int_D|K'(x,y)-K(x,y)|^2\sigma^2(y,u(y))dy\Big)^{\frac{p}{2}}\right]\\
 \leq&\, C(p, L_1, L_2)(\frac{1}{n})^{2p\gamma(d,\epsilon)} [1+E(\|u\|^p_\infty)]\rightarrow0,\quad \mbox{as}\ n\rightarrow\infty.
\end{split}\end{equation*}
The assertion (ii) is proved.

Set $\bar{u}^n(x):=\bar{v}^n(x)+\bar{z}^n(x)$. Since $u(x)=\bar{v}(x)+\bar{z}(x)$ by definition, it follows from (\ref{eq:sec5con1}) and (\ref{eq:sec5con2}) that
\begin{equation}\label{eq:sec5con3}
 \lim\limits_{n\to\infty}E[\sup\limits_{x\in D}|\bar{u}^n(x)-u(x)|^p]=0.
\end{equation}
Hence, to prove the theorem, it is sufficient to show that
\begin{equation}\label{eq:baru_tildeu}
 \lim\limits_{n\to\infty}E[\sup\limits_{x\in D}|\bar{u}^n(x)-\tilde u^n(x)|^p]=0.
\end{equation}
Applying Lemma \ref{le:thm1} to (\ref{eq:sec5ell1}) and (\ref{eq:sec5ell2}), it follows
\begin{equation}\label{eq:sec5in1}\begin{split}
 \sup\limits_{x\in D}|\bar{u}^n(x)-\tilde u^n(x)|
=&\sup\limits_{0\leq k\leq (n-1)^d}|\bar{V}^n_k-V^n_k+\bar{Z}^n_k-Z_k^n|\\
\leq&\, 2\sup\limits_{0\leq k\leq (n-1)^d}|\bar{V}^n_k-V^n_k|
\leq 2\sup\limits_{x\in D}|\bar{v}^n(x)-v^n(x)|.
\end{split}\end{equation}
Set
\begin{equation*}
\hat{v}^n(x):=\int_D K^n(x,y)f(\underline{k}_n(y),\bar{u}^n(\underline{k}_n(y)))dy+\int_D K^n(x,y)\sigma(\underline{k}_n(y),\bar{u}^n(\underline{k}_n(y)))W(dy).
\end{equation*}

To estimate the difference $\bar{v}^n-v^n$, it is sufficient to estimate the differences $\bar{v}^n-\hat{v}^n$ and $\hat{v}^n-v^n$. Recalling the expression of $v^n$ in \eqref{eq:v^n1}, we have
\begin{equation*}\begin{split}
 \hat{v}^n(x)-v^n(x)
=&\int_D K^n(x,y)[f(\underline{k}_n(y),\bar{u}^n(\underline{k}_n(y)))-f(\underline{k}_n(y),u^n(\underline{k}_n(y)))]dy\\
&+\int_D K^n(x,y)[\sigma(\underline{k}_n(y),\bar{u}^n(\underline{k}_n(y)))-\sigma(\underline{k}_n(y),u^n(\underline{k}_n(y)))]W(dy).
\end{split}\end{equation*}

In view of formula \eqref{eq:Kestimate}and Lemma \ref{le:kolmogorov}, we have
\begin{equation}\label{eq:sec5in2}\begin{split}
 &E\left[\sup\limits_{x\in D}|\hat{v}^n(x)-v^n(x)|^p\right]\\
\leq&\,2^{p-1}E\left[\sup\limits_{x\in D}\left|\int_D K^n(x,y)(f(\bar{u}^n(\underline{k}_n(y)))-f(u^n(\underline{k}_n(y))))dy\right|^p\right]\\
&+2^{p-1}E\left[\sup\limits_{x\in D}\left|\int_D K^n(x,y)(\sigma(\bar{u}^n(\underline{k}_n(y)))-\sigma(u^n(\underline{k}_n(y))))W(dy)\right|^p\right]\\
\leq&\,[2^{p-1}L_1^p(\tilde C_D)^{\frac{p}{2}}+2^{2p-2}c_pL_1^p(aB^{p\over 2}+\tilde C_D^{p\over 2})]E\left[\sup\limits_{x\in D}|\bar{u}^n(x)-u^n(x)|^p\right].
\end{split}\end{equation}

Denote $\tilde C':=2^{3p-2}L_1^p(\tilde C_D)^{\frac{p}{2}}+2^{4p-3}c_pL_1^p(aB^{p\over 2}+\tilde C_D^{p\over 2})<1.$
Combining formulas \eqref{eq:sec5in1} and \eqref{eq:sec5in2}, we have
\begin{equation*}\begin{split}
&E\left[\sup\limits_{x\in D}|\bar{u}^n(x)-u^n(x)|^p\right]\\
\leq&\, 2^pE\left[\sup\limits_{x\in D}|\bar{v}^n(x)-v^n(x)|^p\right]\\
\leq&\, 2^{2p-1}E\left[\sup\limits_{x\in D}|\bar{v}^n(x)-\hat{v}^n(x)|^p\right]+2^{2p-1}E\left[\sup\limits_{x\in D}|\hat{v}^n(x)-v^n(x)|^p\right]\\
\leq&\, 2^{2p-1}E\left[\sup\limits_{x\in D}|\bar{v}^n(x)-\hat{v}^n(x)|^p\right]+\tilde C'E\left[\sup\limits_{x\in D}|\bar{u}^n(x)-u^n(x)|^p\right].
\end{split}\end{equation*}

Hence there exists a constant $c''$ such that
\begin{equation*}
 E[\sup\limits_{x\in D}|\bar{u}^n(x)-u^n(x)|^p]\leq c''E[\sup\limits_{x\in D}|\bar{v}^n(x)-\hat{v}^n(x)|^p].
\end{equation*}

Finally, it remains to estimate the difference $\bar{v}^n-\hat{v}^n$. By the integral formulations of $\hat{v}^n$ and $\bar{v}^n$, we have
\begin{equation}\label{eq:sec5re1}\begin{split}
  \hat{v}^n(x)-\bar{v}^n(x)
 =&\int_D(K^n(x,y)-K'(x,y))f(\underline{k}_n(y),\bar{u}^n(\underline{k}_n(y)))dy\\
&+\int_D(K^n(x,y)-K'(x,y))\sigma(\underline{k}_n(y),\bar{u}^n(\underline{k}_n(y)))W(dy)\\
&+\int_DK'(x,y)(f(\underline{k}_n(y),\bar{u}^n(\underline{k}_n(y)))-f(y,u(y)))dy\\
&+\int_DK'(x,y)(\sigma(\underline{k}_n(y),\bar{u}^n(\underline{k}_n(y)))-\sigma(y,u(y)))W(dy)\\
:=& B_1^n(x)+B_2^n(x)+B_3^n(x)+B_4^n(x).
\end{split}\end{equation}
We will show that each of the four terms tends to zero.

In view of \eqref{eq:K'Kdifference}, we have
\begin{equation*}
 \sup\limits_{x\in D}\int_D(K(x,y)-K'(x,y))^2dy\to 0, \quad \text{as}\ n\to\infty.
\end{equation*}
In view of \cite[Lemma 3.5]{GYO}, for every $\epsilon>0$, there exists a constant $C(d,\epsilon)$, such that for $x\in D$ and $n\in\mathbb{N}^*$,
\begin{equation}\label{eq:KKndifference}
 \int_D|K(x,y)-K^n(x,y)|^2dy\leq C(d,\epsilon)n^{-2\sigma(d,\epsilon)},
\end{equation}
where
\begin{equation*}
 \sigma(d,\epsilon)=\begin{cases}
  1-\epsilon,\quad& \text{if}\ d=1,\\
  \frac{1}{2}-\epsilon,\quad& \text{if}\ d=2,\\
  \frac{1}{5}-\epsilon,\quad& \text{if}\ d=3.
 \end{cases}
\end{equation*}
We also need the following estimates on $u$ and $\bar{u}^n$: in view of Lemma \ref{le:thm1} we have
\begin{equation*}
  E[\sup\limits_{x\in D}|\bar{u}^n(x)|^p]\leq E[\sup\limits_{x\in D}|\bar{v}^n(x)|^p+\sup\limits_{x\in D}|\bar{z}^n(x)|^p]\leq 2E[\sup\limits_{x\in D}|\bar{v}^n(x)|^p];
\end{equation*}
in view of Lemma \ref{le:kolmogorov}, formulas \eqref{eq:Kestimate} and \eqref{eq:Klipschitz}, there exists a constant $C$ such that
\begin{equation*}\begin{split}
  E[\sup\limits_{x\in D}|\bar{v}^n(x)|^p]
 \leq&\, CE[\sup\limits_{x\in D}|\int_DK'(x,y)f(y,u(y))dy|^p]\\
 &+CE[\sup\limits_{x\in D}|\int_DK'(x,y)\sigma(y,u(y))W(dy)|^p]\\
 \leq&\, C E\left[\sup\limits_{x\in D}(1+|u(x)|^2)^{\frac{p}{2}}\right]<\infty.
\end{split}\end{equation*}

Then, in view of formulas \eqref{eq:K'Kdifference} and \eqref{eq:KKndifference}, we have
\begin{equation}\label{eq:sec5b1}\begin{split}
 E\sup\limits_{x\in D}|B_1^n(x)|^p
&\leq E[\int_D(\sup\limits_{x\in D}|K^n(x,y)-K'(x,y)|)^2dy]^{\frac{p}{2}}\cdot[\int_D|f(\underline{k}_n(y),\bar{u}^n(\underline{k}_n(y)))|^2dy]^{\frac{p}{2}}\\
&\leq [\int_D(\sup\limits_{x\in D}|K^n(x,y)-K'(x,y)|)^2dy]^{\frac{p}{2}}\cdot[\int_D 1+E[|\bar{u}^n(\underline{k}_n(y))|^2]dy]^{\frac{p}{2}}\\
&\leq C\sup\limits_{x\in D} [\int_D|K^n(x,y)-K'(x,y)|^2dy]^{\frac{p}{2}}
\to 0, \quad  \text{as}\ n\to\infty.
\end{split}\end{equation}
Similar to the proof of formula \eqref{eq:sec5con2}, it is sufficient to prove $E[|B_2^n(x)|^p]\to 0,\, as\ n\to\infty$. In fact
\begin{equation}\label{eq:sec5b2}\begin{split}
&E[|B_2^n(x)|^p]\\
&\leq c_pE\left[\left(\int_D|K^n(x,y)-K'(x,y)|^2|\sigma(\underline{k}_n(y),\bar{u}^n(\underline{k}_n(y)))|^2dy\right)^{\frac{p}{2}}\right]\\
&\leq C E\left[\left(\int_D|K^n(x,y)-K'(x,y)|^2(1+|\bar{u}^n(\underline{k}_n(y))|^2)dy\right)^{\frac{p}{2}}\right]\\
&\leq C \left[1+E\left(\sup\limits_{y\in D}|\bar{u}^n(\underline{k}_n(y))|^p\right)\right]\cdot\left[\int_D|K^n(x,y)-K'(x,y)|^2dy\right]^{\frac{p}{2}}\\
&\to\ 0, \quad  \text{as}\ n\to\infty.
\end{split}\end{equation}
Similarly, we also have
\begin{equation}\label{eq:sec5b3}
 \lim\limits_{n\to\infty} E(|B_3^n(x)|^p+|B_4^n(x)|^p)=0.
\end{equation}
Putting formulas \eqref{eq:sec5b1}, \eqref{eq:sec5b2} and \eqref{eq:sec5b3} together, we obtain
\begin{equation}\label{eq:finalestin}
  \lim\limits_{n\to\infty}E\left[\sup\limits_{x\in D} |\hat{v}^n(x)-\bar{v}^n(x)|^p \right]=0.
\end{equation}
Finally we complete the proof of formula \eqref{eq:baru_tildeu} and hence the theorem.
\end{proof}

\section{Acknowledge}
The first author is sincerely grateful to his PhD advisor, Professor Shanjian Tang, for giving many useful suggestions and comments on the improvements. The first author also would like to thank Professor Ralf Kornhuber for inspiring discussions on numerical PDE theory.


\end{document}